\documentclass[a4paper,11pt,oneside]{article}
\usepackage{amsfonts,amsmath,amsthm,scalerel,verbatim}
\newcommand{\Z}{\mathbb{Z}}
\makeatletter
\newcommand{\vast}{\bBigg@{4}}
\newcommand{\Vast}{\bBigg@{5}}
\makeatother
\newcommand{\pres}[2]{\langle {#1}\ |\ {#2} \rangle}

\newcommand{\onetwopres}[3]{\Bigg \langle {#1}\ \Bigg |\
\begin{array}{l}
{#2}\\{#3}
\end{array}\Bigg \rangle}
\newcommand{\onethreepres}[4]{\Bigg \langle {#1}\ \Bigg |\
\begin{array}{l}
{#2}\\{#3}\\{#4}
\end{array}\Bigg \rangle}

\newcommand{\threethreepres}[6]{\Bigg \langle \begin{array}{l}
{#1}\\{#2}\\{#3}
\end{array}
\ \Bigg |\
\begin{array}{l}
{#4}\\{#5}\\{#6}
\end{array}
\Bigg \rangle}
\newcommand{\twothreepres}[5]{\Bigg \langle \begin{array}{l}
{#1}\\{#2}
\end{array}
\ \Bigg |\
\begin{array}{l}
{#3}\\{#4}\\{#5}
\end{array}
\Bigg \rangle}

\newtheorem{theorem}{Theorem}
\newtheorem{lemma}[theorem]{Lemma}
\newtheorem{corollary}[theorem]{Corollary}
\newtheorem{remark}[theorem]{Remark}
\newtheorem{conjecture}[theorem]{Conjecture}
\newtheorem{maintheorem}{Theorem}

\newtheorem{maincorollary}[maintheorem]{Corollary}

\allowdisplaybreaks

\begin{document}
\title{Generalized Fibonacci groups $H(r,n,s)$ that are connected Labelled Oriented Graph groups}%
\author{Gerald Williams}

\maketitle

\begin{abstract}
The class of connected LOG (Labelled Oriented Graph) groups coincides with the class of fundamental groups of complements of closed, orientable 2-manifolds embedded in $S^4$, and so contains all knot groups. We investigate when Campbell and Robertson's generalized Fibonacci groups $H(r,n,s)$ are connected LOG groups. In doing so, we use the theory of circulant matrices to calculate the Betti numbers of their abelianizations. We give an almost complete classification of the groups $H(r,n,s)$ that are connected LOG groups. All torus knot groups and the infinite cyclic group arise and we conjecture that these are the only possibilities. As a corollary we show that $H(r,n,s)$ is a 2-generator knot group if and only if it is a torus knot group.
\end{abstract}

\section{Introduction}\label{sec:intro}

A \em Labelled Oriented Graph (LOG) \em consists of a finite graph (possibly with loops and multiple
edges) with vertex set $V$ and edge set $E$ together with three maps $\iota,\tau,\lambda: E \rightarrow V$ called the \em initial
vertex map, terminal vertex map, \em and \em labelling map, \em respectively. The LOG determines a corresponding
\em LOG presentation\em
\[ \pres{V}{\tau(e)^{-1}\lambda(e)^{-1}\iota(e)\lambda(e)\ (e\in E)}.\]
A group with a LOG presentation is called a \em LOG group\em~\cite{Howie85}. When the underlying graph is connected we have a \em connected LOG, \em a \em connected LOG presentation\em, and a \em connected LOG group\em. A \em $k$-knot group \em ($k\geq 0$) is the fundamental group of the complement of an $k$-sphere $S^k$ in $S^{k+2}$; we refer to a $1$-knot group as a \em knot group\em. The Wirtinger presentation of a knot group is a connected LOG presentation and so all knot groups are connected LOG groups; in particular, the infinite cyclic group is a connected LOG group. Further examples of LOG groups include right angled Artin groups and braid groups. Clearly the abelianisation of a LOG group is torsion-free, and the abelianisation of a connected LOG groups is the infinite cyclic group.

As pointed out in~\cite{GilbertHowie95},\cite{Gilbert17}, by~\cite{Simon80} a group is a connected LOG group if and only if it is the fundamental group of the complement of a closed, orientable 2-manifold embedded in $S^4$. It follows that every $k$-knot group is a connected LOG group~(\cite{AcunaGordonSimon10}). A particular instance of~\cite[Theorem~1.1]{AcunaGordonSimon10} is that, given a finite presentation of a group, there is no algorithm that can decide if that group is a connected LOG group. Another is that, given a finite presentation of a group, there is no algorithm that can decide if that group is a knot group (see also~\cite[Theorem~9.2.1]{Neuwirth},\cite{Stallings62}).

A \em cyclically presented group \em is a group defined by a presentation of the form
\begin{alignat*}{1}
G_n(w)=\pres{x_0,\ldots ,x_{n-1}}{w(x_i,x_{i+1},\ldots ,x_{i+n-1})\ (0\leq i<n)}
\end{alignat*}
where $w=w(x_0,x_1,\ldots ,x_{n-1})$ is some fixed element of the free group $F(x_0,\ldots ,x_{n-1})$. Connections between HNN extensions of cyclically presented groups and LOG groups have been investigated in~\cite{GilbertHowie95},\cite{SzczepanskiVesnin01},\cite{HowieWilliams12}. Asphericity of certain cyclic presentations that are (connected) \em word labelled oriented graph (WLOG) presentations \em is established in~\cite[Section~3]{HarlanderRosebrock15}. In this paper we investigate a particular family of cyclically presented groups and aim to classify when they are connected LOG groups or when they are knot groups. Namely, we investigate the \em generalized Fibonacci groups \em
\begin{alignat*}{1}
H(r,n,s)=\pres{x_0,\ldots,x_{n-1}}{\prod_{j=0}^{r-1} x_{i+j}=\prod_{j=0}^{s-1} x_{i+j+r}\ (0\leq i<n)}\label{eq:Hrns}
\end{alignat*}
where $r,s\geq 1$, $n\geq 2$, and subscripts are taken mod~$n$, that were introduced in~\cite{CampbellRobertson75}. Setting $r=2,s=1$ we get the \em Fibonacci groups \em $F(2,n)$ introduced in~\cite{Conway65}.

These groups have been considered from algebraic and topological perspectives. Finite groups $H(r,n,s)$ have been obtained in~\cite{CampbellRobertson75},\cite{Brunner77},\cite[Corollary~E]{BogleyWilliams16},\cite[Corollary~11]{BogleyWilliamsCoherence}; conditions under which $H(r,n,s)$ is infinite are given in~\cite{CampbellRobertson75},\cite{CampbellThomas86}; and conditions under which $H(r,n,s)$ is large, SQ-universal, or contains a non-abelian free subgroup can be extracted from~\cite{Williams12}. Asphericity of the presentations $H(r,n,s)$ is considered in~\cite[Theorem~3]{Prischcepov95}. A class of groups $H(r,n,s)$ that are fundamental groups of closed 3-manifolds was obtained in~\cite[Proposition~3]{SV00}. In the opposite direction conditions under which $H(r,n,s)$ is not the fundamental group of a hyperbolic 3-dimensional orbifold of finite volume are given in~\cite[Corollary~3.2]{CRS03}. Corollary~5.5 of~\cite{CRS03} gives that, for $r\neq s$, the natural HNN extension of the group $H(r,n,s)$ is a 3-knot group if and only if $|r-s|=1$. Note that by inverting the relators, replacing each generator by its inverse, and negating the subscripts we have that $H(r,n,s)\cong H(s,n,r)$.

Our main result is the following (recall that a group $G$ is \em perfect \em if $G^\mathrm{ab}=1$).

\begin{maintheorem}\label{thm:connLOGclassification}
Let $r,s\geq 1$, $n\geq 2$. If  $H(r,n,s)$ is a connected LOG group then one of the following holds:
\begin{itemize}
  \item[(a)] $r=s$ and $(r,n)=1$ in which case $H(r,n,s)\cong \pres{a,b}{a^n=b^r}$, the fundamental group of the complement of the $(r,n)$-torus knot in~$S^3$;
  \item[(b)] $(r,n,s)=2$, $|r-s|=2$ and either $r\equiv0$~mod~$n$ or $s\equiv0$~mod~$n$, in which case $H(r,n,s)\cong \Z$;
  \item[(c)] $(r,n,s)=2$, $|r-s|=2$, $\{r,s\}\neq \{4,2\}$, $(n,r+s)=2$, $r\not \equiv 0$~mod~$n$, $s\not \equiv0$~mod~$n$, and the group $H(r/2,n/2,s/2)$ is perfect.
\end{itemize}
\end{maintheorem}

We conjecture that condition~(c) cannot hold.

\begin{conjecture}\label{conj:noperfect}
Let $r,s\geq 1$, $n\geq 2$, $r\not \equiv 0$~mod~$n$, $s\not \equiv 0$~mod~$n$. Then $H(r,n,s)^\mathrm{ab}\neq 1$.
\end{conjecture}

Knots (i.e. complements of $S^1$ in $S^3$) for which the minimum number of generators required to generate the corresponding knot group is equal to two are called \em 2-generator knots \em and the corresponding knot group is a \em 2-generator knot group. \em Since the only knot for which the corresponding group is cyclic is the unknot, 2-generator knots are, from one perspective, the `simplest' non-trivial knots. All tunnel number one knots (in particular all torus knots) are 2-generator knots, and it has been conjectured that all 2-generator knots are tunnel number one knots  (in~\cite{Bleiler94} this is attributed to Scharlemann~\cite{Scharlemann84}, who attributes it to Casson; see also~\cite[Conjecture~3.9]{Ozawa16}). The conjecture has been shown to hold for cable knots~\cite{Bleiler94}
and the satellite knots that have a two-generator presentation in which at least one generator is represented by a meridian for the knot are classified in~\cite{BleilerJones}. As a corollary to Theorem~\ref{thm:connLOGclassification} we classify when $H(r,n,s)$ is a 2-generator knot group.

\begin{maincorollary}\label{cor:2genknot}
Let $r,s\geq 1$, $n\geq 2$. Then $H(r,n,s)$ is a 2-generator knot group if and only if $r=s$ and $(r,n)=1$, in which case $H(r,n,s)\cong\pres{a,b}{a^n=b^r}$, the fundamental group of the complement of the $(r,n)$-torus knot in~$S^3$.
\end{maincorollary}

Since connected LOG groups abelianize to the infinite cyclic group $\Z$, the abelianisation of $H(r,n,s)$ is of interest to us. Any finitely generated abelian group $A$ is isomorphic to a group of the form $A_0\oplus \Z^\beta$ where $A_0$ is a finite abelian group and $\beta\geq 0$. The number $\beta=\beta(A)$ is called the \em Betti number \em (or \em torsion-free rank\em) of $A$, and we write $d(A)$ to denote the minimum number of generators of $A$. Clearly $A$ is infinite if and only if $\beta(A)\geq 1$ and if $G$ is a connected LOG group then $\beta(G^\mathrm{ab})=d(G^\mathrm{ab})=1$. Theorem~1 of~\cite{CampbellRobertson75} asserts that for $r\neq s$ we have that $\beta (H(r,n,s)^\mathrm{ab})\geq 1$ if and only if the greatest common divisor $(r,n,s)\geq 2$. In Theorem~\ref{thm:betti} we generalize this to give the value of $\beta(H(r,n,s)^\mathrm{ab})$ in all cases.

\begin{maintheorem}\label{thm:betti}
Let $r,s\geq 1$, $n\geq 2$.
\begin{itemize}
  \item[(a)] If $r\neq s$ then $\beta(H(r,n,s)^\mathrm{ab})=(r,n,s)-1$;
  \item[(b)] if $r=s$ then $\beta(H(r,n,s)^\mathrm{ab})=(r,n)$.
\end{itemize}
\end{maintheorem}

In support of Conjecture~\ref{conj:noperfect} we have the following.
\begin{corollary}\label{cor:finitelymany}
Let $r\geq 1$ (resp. $s\geq 1$). Then there are at most finitely many values of $n\geq 2$, $s\geq 1$ (resp. $r\geq 1$) such that $H(r,n,s)^\mathrm{ab}=1$.
\end{corollary}

\begin{remark}\label{rem:finitelymany}
\em It is not hard to prove that $H(r,n,s)^\mathrm{ab}\cong H(r+\alpha n,n,s+\alpha n)^\mathrm{ab}$ for all $\alpha\geq 0$, so if there is a choice of $r,s\geq 1$, $n\geq 2$, $r\not \equiv 0$~mod~$n$, $s\not \equiv 0$~mod~$n$, such that $H(r,n,s)^\mathrm{ab}=1$ then there are infinitely many such choices of $r,n,s$ such that $H(r,n,s)^\mathrm{ab}=1$. Therefore Corollary~\ref{cor:finitelymany} does not imply that there are at most finitely many $r,s\geq 1$, $n\geq 2$ such that $H(r,n,s)^\mathrm{ab}=1$.
\end{remark}
The proofs of Theorem~\ref{thm:betti} and Corollary~\ref{cor:finitelymany} use the theory of circulant matrices. The circulant matrix $C=\mathrm{circ}_n(a_0,\ldots, a_{n-1})$ is the $n\times n$ matrix whose first row is $(a_0,\ldots, a_{n-1})$ and where each subsequent row is a cyclic shift of its predecessor by one column. Thus if, for each $0\leq i<n$, the exponent sum of $x_i$ in $w(x_0,\ldots ,x_{n-1})$ is $a_i$ then the relation matrix of $G_n(w)$ is the circulant matrix $C$. The \em representer polynomial \em of $C$ is the polynomial
\[ f(t)=\sum_{i=0}^{n-1}a_it^i\]
and we define $g(t)=t^n-1$. It is well known that
\begin{alignat}{1}
 |\mathrm{det} (C)|=\Big| \prod_{g(\lambda)=0}f(\lambda)\Big|\label{eq:detC}
\end{alignat}
and so this is the order $|G_n(w)^\mathrm{ab}|$ when it is non-zero, and $G_n(w)^\mathrm{ab}$ is infinite otherwise. This fact has long been used in the theory of cyclically presented groups (see~\cite{Johnson80}) and, in particular, it was used to obtain~\cite[Theorem~1]{CampbellRobertson75}. The rank of $C$ can also be expressed in terms of the polynomials $f,g$; specifically
\begin{alignat}{1}
 \mathrm{rank}(C)=n-\mathrm{deg}(\mathrm{gcd}(f(t),g(t)))\label{eq:rank}
\end{alignat}
where $\mathrm{deg}(\cdot)$ denotes the degree (see~\cite[Proposition~1.1]{Ingleton56} or~\cite[Theorem~1]{Newman83}) and so
\begin{alignat}{1}
\beta(G_n(w)^\mathrm{ab})=\mathrm{deg}(\mathrm{gcd}(f(t),g(t)))\label{eq:betti}
\end{alignat}
and this is the engine of the proof of Theorem~\ref{thm:betti}. While the formula~(\ref{eq:rank}) is old, we believe that it has not been applied to cyclically presented groups before. Moreover, we expect it to be of independent interest in studying other classes of cyclically presented groups, and with wider applications than that considered here. The proof of Corollary~\ref{cor:finitelymany} uses a result from~\cite{Cremona08} for determining when $|\mathrm{det}(C)|=1$; that is, when $C$ is a unimodular matrix.

\section{A class of groups $H(r,n,s)$ that are knot groups}\label{sec:knotgroups}

If $n$ is odd then
\begin{alignat}{1}
H(2,n,2)=\pres{x_0,\ldots, x_{n-1}}{x_ix_{i+1}=x_{i+2}x_{i+3}\ (0\leq i <n)}.\label{eq:H2n2}
\end{alignat}
The $n-1$ relations $x_ix_{i+1}=x_{i+2}x_{i+3}$  ($0\leq i < n-1$) imply the $n$'th relation $x_{n-1}x_0=x_1x_2$, so this redundant relation may be eliminated to give the $n$ generator, $n-1$ relation presentation
\[ \pres{x_0,\ldots, x_{n-1}}{x_ix_{i+1}=x_{i+2}x_{i+3}\ (0\leq i < n-1)}\]
of $H(2,n,2)$. This is precisely the Dehn presentation for the $(2,n)$-torus knot (see, for example, \cite[page~155]{GilbertPorter}, where the case $n=5$ is illustrated).

Furthermore, if $n$ is odd, then the following sequence of relations are implied by the relations of~(\ref{eq:H2n2})
\[ x_ix_{i+1}=x_{i+2}x_{i+3}=x_{i+4}x_{i+5}=\cdots =x_{i+n-3}x_{i+n-2}=x_{i-1}x_i\]
and, in particular, $x_ix_{i+1}=x_{i-1}x_i$. Conversely, the relations $x_ix_{i+1}=x_{i-1}x_i$ ($0\leq i<n$) imply the sequence of relations
\[ x_ix_{i+1}=x_{i-1}x_i= x_{i-2}x_{i-1}=x_{i-3}x_{i-2}=\cdots = x_{i+2}x_{i+3}\]
and, in particular, $x_ix_{i+1}=x_{i+2}x_{i+3}$. Thus the group $H(2,n,2)$ is also given by the presentation
\[\pres{x_0,\ldots ,x_{n-1}}{x_ix_{i+1}=x_{i-1}x_i\ (0\leq i<n)}\]
which is a (cyclic) Wirtinger presentation for the $(2,n)$-torus knot (\cite[pages~151--153]{GilbertPorter}) arising from a LOG where the underlying graph is a cycle.

Thus if $n$ is odd we have that $H(2,n,2)$ is the fundamental group of the complement of the $(2,n)$-torus knot in $S^3$. More generally we have the following.

\begin{lemma}\label{lem:(r,n)=1knot}
Let $r\geq 1$, $n\geq 2$. If $(r,n)=1$ then $H(r,n,r)\cong \pres{a,b}{a^n=b^r}$, the fundamental group of the complement of the $(r,n)$-torus knot in $S^3$.
\end{lemma}

\begin{proof}
Since $(r,n)=1$ there exist $\gamma,\delta $ such that $\gamma r-\delta n=1$. Let $H=H(r,n,r)$, then
\begin{alignat*}{1}
H
&=\threethreepres
{x_0,\ldots ,x_{n-1},}{a_0,\ldots ,a_{n-1},}{b_0,\ldots ,b_{n-1}}
{x_ix_{i+1}\ldots x_{i+r-1}=x_{i+r}x_{i+r+1}\ldots x_{i+2r-1},}{a_i=x_ix_{i+1}\ldots x_{i+r-1},}{b_i=x_ix_{i+1}\ldots x_{i+n-1}\ (0\leq i<n)}\\
&=\threethreepres
{x_0,\ldots ,x_{n-1},}{a_0,\ldots ,a_{n-1},}{b_0,\ldots ,b_{n-1}}
{a_i=a_{i+r},}{a_i=x_ix_{i+1}\ldots x_{i+r-1},}{b_i=x_ix_{i+1}\ldots x_{i+n-1}\ (0\leq i<n)}.
\end{alignat*}
Now
\[ a_ia_{i+r}a_{i+2r}\ldots a_{i+(\gamma -1)r}=(x_ix_{i+1}\ldots x_{i+n-1})^\delta x_i=b_i^\delta x_i\]
and
\[ x_i^{-1}b_ix_i = x_{i+1}x_{i+2}\ldots x_{i+n-1}x_{i}=b_{i+1}\]
so we may add the relations $a_ia_{i+r}a_{i+2r}\ldots a_{i+(\gamma -1)r}=b_i^\delta x_i$ and $x_i^{-1}b_ix_i=b_{i+1}$ to get
\begin{alignat*}{1}
H
&=\threethreepres
{x_0,\ldots ,x_{n-1},}{a_0,\ldots ,a_{n-1},}{b_0,\ldots ,b_{n-1}}
{a_i=a_{i+r},\ a_ia_{i+r}a_{i+2r}\ldots a_{i+(\gamma -1)r}=b_i^\delta x_i,}{a_i=x_ix_{i+1}\ldots x_{i+r-1},\ x_i^{-1}b_ix_i=b_{i+1},}{b_i=x_ix_{i+1}\ldots x_{i+n-1}\ (0\leq i<n)}.\\
\intertext{Now $(r,n)=1$ so the sequence of equalities $a_0=a_r=a_{2r}=\cdots =a_{(n-1)r}=a_0$ includes all $a_i$, so $a_i=a_0$ for all $i$ and so we can eliminate $a_1,\ldots ,a_{n-1}$ and write $a=a_0$ to get}
H
&=\twothreepres
{x_0,\ldots ,x_{n-1},}{a,b_0,\ldots ,b_{n-1}}
{a^\gamma =b_i^\delta x_i,}{a=x_ix_{i+1}\ldots x_{i+r-1},\ x_i^{-1}b_ix_i=b_{i+1},}{b_i=x_ix_{i+1}\ldots x_{i+n-1}\ (0\leq i<n)}\\
&=\twothreepres
{x_0,\ldots ,x_{n-1},}{a,b_0,\ldots ,b_{n-1}}
{x_i=b_i^{-\delta }a^\gamma, }{a=x_ix_{i+1}\ldots x_{i+r-1},\ x_i^{-1}b_ix_i=b_{i+1},}{b_i=x_ix_{i+1}\ldots x_{i+n-1}\ (0\leq i<n)}\\
&=\twothreepres
{x_0,\ldots ,x_{n-1}}{a,b_0,\ldots ,b_{n-1},}
{x_i=b_i^{-\delta }a^\gamma, }{a=x_ix_{i+1}\ldots x_{i+r-1},\ a^{-\gamma }b_ia^\gamma =b_{i+1},}{b_i=x_ix_{i+1},\ldots x_{i+n-1}\ (0\leq i<n)}\\
&=\onethreepres
{a,b_0,\ldots ,b_{n-1}}
{a=(b_i^{-\delta }a^\gamma )  (b_{i+1}^{-\delta }a^\gamma ) \ldots (b_{i+r-1}^{-\delta }a^\gamma ),}{a^{-\gamma }b_ia^\gamma =b_{i+1},}
{b_i=(b_i^{-\delta }a^\gamma ) (b_{i+1}^{-\delta }a^\gamma ) \ldots (b_{i+n-1}^{-\delta }a^\gamma )\ (0\leq i<n)}\\
&=\onetwopres
{a,b_0,\ldots ,b_{n-1}}
{a=\prod_{j=0}^{r-1}b_{i+j}^{-\delta }a^\gamma ,\ a^{-\gamma }b_ia^\gamma =b_{i+1},}
{b_i=\prod_{k=0}^{n-1}b_{i+k}^{-\delta }a^\gamma \ (0\leq i<n)}.
\end{alignat*}
The relations $a^{-\gamma }b_ia^\gamma =b_{i+1}$ imply $a^{-\gamma }b_i^{-\delta }a^\gamma =b_{i+1}^{-\delta }$ so  $b_i^{-\delta }a^\gamma =a^{\gamma }b_{i+1}^{-\delta }$. Therefore
\[ \prod_{j=0}^{r-1} b_{i+j}^{-\delta }a^\gamma =a^{\gamma r } b_{i+r}^{-\delta r }\]
so the relation $a=\prod_{j=0}^{r-1} b_{i+j}^{-\delta }a^\gamma $ is equivalent to $a=a^{\gamma r} b_{i+r}^{-\delta r }$ which is equivalent to $a^{\gamma r -1}=b_{i+r}^{\delta r }$, i.e.\,to $a^{\delta n }=b_{i+r}^{\delta r }$, so the set of relations
$a=\prod_{j=0}^{r-1} b_{i+j}^{-\delta }a^\gamma $ is equivalent to the set of relations $a^{\delta n }=b_{i}^{\delta r }$ ($0\leq i <n$).
Similarly we have
\[ \prod_{k=0}^{n-1} b_{i+k}^{-\delta }a^\gamma =a^{\gamma n} b_{i+n}^{-\delta n}\]
so the relation $b_i=\prod_{k=0}^{n-1}b_{i+k}^{-\delta }a^\gamma $ is equivalent to $b_i=a^{\gamma n} b_{i}^{-\delta n}$, which is equivalent to $b_i^{1+\delta n}=a^{\gamma n}$, i.e.\,to $b_i^{\gamma r}=a^{\gamma n}$. Thus
\begin{alignat*}{1}
H
&=\onetwopres{a,b_0,\ldots ,b_{n-1}}{b_i^{\delta r }=a^{\delta n }, a^{-\gamma }b_ia^\gamma =b_{i+1},}{b_i^{\gamma r }=a^{\gamma n }\ (0\leq i<n)}\\
&=\onetwopres{a,b_0,\ldots ,b_{n-1}}{b_i^{\delta r }=a^{\delta n }, b_j=a^{-j\gamma }b_0a^{j\gamma }\ (1\leq j<n),}{b_i^{\gamma r }=a^{\gamma n }\ (0\leq i<n), a^{-\gamma n }b_0a^{\gamma n }=b_0}\\
&=\pres{a,b}{b^{\delta r }=a^{\delta n }, b^{\gamma r }=a^{\gamma n }, a^{-\gamma n }ba^{\gamma n }=b}\quad{\mathrm{where}~b=b_0}\\
&=\pres{a,b}{b^{\delta r }=a^{\delta n }, b^{\gamma r }=a^{\gamma n }}.
\end{alignat*}
Now
\begin{alignat*}{1}
  a^n=a^{n\cdot 1}=a^{n(\gamma r-\delta n)}=(a^{\gamma n })^r \cdot (a^{\delta n })^{-n}&=(b^{\gamma r })^r\cdot (b^{\delta r })^{-n}\\
  &\quad \quad =(b^r)^{\gamma r-\delta n}=(b^r)^1=b^r
\end{alignat*}
so we may add the relation $a^n=b^r$ and then eliminate the redundant relations $b^{\delta r }=a^{\delta n }, b^{\gamma r }=a^{\gamma n }$ to get $H=\pres{a,b}{a^n=b^r}$, as required.
\end{proof}

\section{Betti numbers and perfect groups}\label{sec:betti}

In this section we prove Theorem~\ref{thm:betti} and Corollary~\ref{cor:finitelymany}.

\begin{proof}[Proof of Theorem~\ref{thm:betti}]
Let $d=(r,n,s)$, $R=r/d$, $N=n/d$, $S=s/d$. The representer polynomial of $H(r,n,s)$ is
\begin{alignat}{1}
f(t)=1+t+\ldots+t^{r-1}-t^{r}-t^{r+1}-\ldots-t^{r+s-1}.\label{eq:fforHrns}
\end{alignat}
Since $H(r,n,s)\cong H(s,n,r)$ we may assume $s\geq r$. If $s=r$ then
\begin{alignat*}{1}
f(t)&=(1-t^{r})(1+t+\ldots+t^{r-1})\\
&=(1-t^{r})(1+t+\ldots+t^{d-1})(1+t^d+\ldots+t^{(R-1)d});
\intertext{if $s>r$ then}
f(t)&=(1-t^{r})(1+t+t^2+\ldots+t^{r-1})-t^{2r}(1+t+\ldots+t^{s-r-1})\\
&=(1-t^{r})(1+t+t^2+\ldots+t^{d-1})(1+t^d+\ldots+t^{(R-1)d})\\
&\ \quad \quad \quad \quad -t^{2r}(1+t+t^2+\ldots+t^{d-1})(1+t^d+\ldots+t^{(S-R-1)d}).
\end{alignat*}
That is,
\[ f(t)= (1+t+t^2+\ldots+t^{d-1})F(t)\]
where
\begin{alignat}{1}
F(t)&= (1-t^{r})(1+t^d+\ldots+t^{(R-1)d})\label{eq:Fs=r}
\end{alignat}
when $s=r$ and
\begin{alignat}{1}
F(t)&=(1-t^{r})(1+t^d+\ldots+t^{(R-1)d})-t^{2r}(1+t^d+\ldots+t^{(S-R-1)d})\label{eq:Fs>r}
\end{alignat}
when $s>r$.
By~(\ref{eq:betti}) we must find the degree of the highest common factor of $f(t)$ and $g(t)=t^n-1$. Observe that
\begin{alignat*}{1}
g(t)&=(1+t+t^2+\ldots+t^{d-1})G(t)
\end{alignat*}
where
\[G(t)=(1-t)(1+t^d+t^{2d}+\ldots + t^{(N-1)d}).\]
Therefore $(f(t),g(t))= (1+t+t^{2}+\ldots + t^{d-1}) (F(t),G(t))$.

Suppose $r=s$ then $F(t)$ is as given at~(\ref{eq:Fs=r}). Now, writing $\Phi_m(t)$ to denote the $m$th cyclotomic polynomial, we have
\begin{alignat*}{1}
&\ (1+t^d+\ldots+t^{(R-1)d}, 1+t^d+t^{2d}+\ldots + t^{(N-1)d}) \\
&\ \qquad \qquad \qquad \qquad \qquad = \left(\prod_{\delta|R,\delta>1}\Phi_\delta(t^d),\prod_{\delta|N,\delta>1}\Phi_\delta(t^d)\right) =1
\end{alignat*}
since $(R,N)=1$. Therefore $(F(t),G(t))=(1-t^{r},1-t)=1-t$ so $(f(t),g(t))=(1+t+t^{2}+\ldots + t^{d-1})(1-t)=1-t^d$, which is of degree $d=(r,n,s)$, as required.

Suppose then that $s>r$ so $F(t)$ is as given at~(\ref{eq:Fs>r}). We must show $(F(t),G(t))=1$. If $\lambda^d=1$, then $\lambda$ is not a root of $F(t)$ (for otherwise we get a contradiction to $S>R$). Therefore $\lambda$ is a root of  $F(t)$ and of $G(t)$ if and only if it is a root of $(1-t^d)F(t)$ and of $(1-t^d)G(t)$. Assume $\lambda$ is such a root. Then, after simplifying, the equations $(1-\lambda^d)F(\lambda)=0$, $(1-\lambda^d)G(\lambda)=0$ imply
\begin{alignat*}{1}
1-2\lambda^r+\lambda^{r+s}=0,\\
1-\lambda^n=1.
\end{alignat*}
Since $\lambda^n=1$ we have $\lambda=e^{i\theta}$ for some $\theta$, and $1=|\lambda|=\lambda\bar{\lambda}$ so $\lambda^{-1}=\bar{\lambda}$. Taking the complex conjugate of the first equation then gives
\[1-2\lambda^{-r}+\lambda^{-r-s}=0.\]
Multiplying the equations $\lambda^{r+s}=2\lambda^r-1$ and $\lambda^{-(r+s)}=2\lambda^{-r}-1$ and simplifying gives $(\lambda^r-1)^2=0$, so $\lambda^r=1$. Similarly, multiplying the equations $2\lambda^r=1+\lambda^{r+s}$ and $2\lambda^{-r}=1+\lambda^{-(r+s)}$ and simplifying gives $(\lambda^{r+s}-1)^2=0$, so $\lambda^{r+s}=1$ and hence $\lambda^s=1$. Therefore we have $\lambda^r=1,\lambda^n=1,\lambda^s=1$, or equivalently
$(\lambda^d)^R=1,(\lambda^d)^N=1,(\lambda^d)^S=1$. But $(R,N,S)=1$ so $\lambda^d=1$, contradicting the fact that $\lambda$ is a root of $F(t)$.
Therefore $F(t),G(t)$ have no common roots so $(F(t),G(t))=1$, as required.
\end{proof}

\begin{proof}[Proof of Corollary~\ref{cor:finitelymany}]
Let $r\geq 1$. By~(\ref{eq:detC}) if $H(r,n_0,s)^\mathrm{ab}=1$ for some $n_0$ then $f(1)=\pm 1$ (where $f$ is as given at~(\ref{eq:fforHrns})) and so $|r-s|=1$ and in particular $(r,s)=1$ and there are at most two possible values of $s$, namely $s=r\pm 1$. If $f(t)$ has a cyclotomic factor $\Phi_{n_0}(t)$, then this is also a factor of $g(t)=t^{n_0}-1$ so $\mathrm{deg}((f(t),g(t))>0$.  But $(r,n_0,s)=(r,n_0,r\pm 1)=1$ so Theorem~\ref{thm:betti}(a) (or~\cite[Theorem~1]{CampbellRobertson75}) gives that $\beta(H(r,n_0,s)^\mathrm{ab})=0$ and so by~(\ref{eq:betti}) we have $\mathrm{deg}(f(t),g(t))=0$, a contradiction. Therefore $f(t)$ has no cyclotomic factors. Since also $f(0)\neq 0$, Theorem~1 of~\cite{Cremona08} implies that the set of integers $n>r$ such that the relation matrix of the presentation $H(r,n,s)$ has determinant equal to $1$ or $-1$ is finite. Therefore the set of integers $n$ for which $H(r,n,s)^\mathrm{ab}=1$ is finite.
\end{proof}

\section{Minimum number of generators for $H(r,n,s)^\mathrm{ab}$}\label{sec:generatornumber}

\begin{lemma}\label{lem:mapstoZ2^t}
The minimum number of generators
\[d(H(r,n,s)^\mathrm{ab})\geq \begin{cases}
  (n,r+s)& \mathrm{if}~(r+s)/(n,r+s)~\mathrm{is~even},\\
  (n,r+s)-1& \mathrm{if}~(r+s)/(n,r+s)~\mathrm{is~odd}.
\end{cases}\]
Hence if $H(r,n,s)$ is a connected LOG group then $(n,r+s)=1$ if $(r+s)/(n,r+s)$ is even and $(n,r+s)\leq 2$ if $(r+s)/(n,r+s)$ is odd.
\end{lemma}

\begin{proof}
Using an idea from the proof of~\cite[Lemma~4]{CampbellRobertson75}, we see that the group $H(r,n,s)^\mathrm{ab}$ maps onto
\begin{alignat*}{1}
Q
&=\pres{x_0,\ldots,x_{n-1}}{\prod_{j=0}^{r-1} x_{i+j}=\prod_{j=0}^{s-1}x_{i+j+r},\ x_i^2\ (0\leq i<n)}^\mathrm{ab}\\
&=\pres{x_0,\ldots,x_{n-1}}{\prod_{j=0}^{r+s-1} x_{i+j},\ x_i^2\ (0\leq i<n)}^\mathrm{ab}.
\end{alignat*}
The remainder of the proof is similar to that of~\cite[Theorem~C]{WilliamsSemigroups}. Let $\delta=(n,r+s)$. Then there exist $p, q \in\Z$ such that $\delta=p(r+s) + qn$ so $\delta\equiv p(r+s)$~mod~$n$. The relation
$x_ix_{i+1}\ldots x_{i+r+s-1}= 1$ implies $x_i(x_{i+1} \ldots x_{i+r+s}) = x_{i+r+s}$ so $x_i = x_{i+r+s}$ and hence
$x_i = x_{i+r+s} = x_{i+2(r+s)} = \ldots = x_{i+p(r+s)}$. But $x_{i+p(r+s)} = x_{i+\delta}$ so we have $x_i = x_{i+\delta}$
for each $0\leq i<n$. Eliminating generators $x_\delta, \ldots, x_{n-1}$ gives
\begin{alignat*}{1}
Q
&=\pres{x_0,\ldots,x_{\delta-1}}{(x_ix_{i+1}\ldots x_{i+\delta-1})^{(r+s)/\delta}, x_i^2\ (0\leq i<\delta)}^\mathrm{ab}\\
&=\pres{x_0,\ldots,x_{\delta-1}}{(x_0x_{1}\ldots x_{\delta-1})^{(r+s)/\delta}, x_i^2\ (0\leq i<\delta)}^\mathrm{ab}\\
&=\pres{x_0,\ldots,x_{\delta-1},y}{y^{(r+s)/\delta},   y=x_0x_{1}\ldots x_{\delta-1}, x_i^2\ (0\leq i<\delta)}^\mathrm{ab}\\
&=\pres{x_0,\ldots,x_{\delta-2},y}{y^{(r+s)/\delta},   y^2, x_i^2\ (0\leq i<\delta-1)}^\mathrm{ab}\\
&=\pres{x_0,\ldots,x_{\delta-2},y}{y^{((r+s)/\delta,2)}, x_i^2\ (0\leq i<\delta-1)}^\mathrm{ab}\\
&\cong \underbrace{\Z_2*\cdots*\Z_2}_{\kappa}
\end{alignat*}
where $\kappa=\delta$ if $(r+s)/\delta$ is even and $\kappa=\delta-1$ otherwise. Hence $H(r,n,s)^\mathrm{ab}$ maps onto $\Z_2^\kappa$ so $d(H(r,n,s)^\mathrm{ab})\geq \kappa$. If $H(r,n,s)$ is a connected LOG group then $d(H(r,n,s)^\mathrm{ab})=1$, and the result follows.
\end{proof}

In connection with Conjecture~\ref{conj:noperfect} we record the following:

\begin{corollary}\label{cor:necperfectconds}
If $H(r,n,s)$ is perfect then $|r-s|=1$ and $(n,r+s)=1$.
\end{corollary}

\begin{proof}
If $H(r,n,s)^\mathrm{ab}$ is perfect then $f(1)=\pm 1$, where $f$ is the representer polynomial for $H(r,n,s)$ given at~(\ref{eq:fforHrns}); that is $|r-s|=1$. By Lemma~\ref{lem:mapstoZ2^t} if $(n,r+s)>1$ then $H(r,n,s)$ is not perfect.
\end{proof}

\begin{remark}\label{rem:searchforperfect}
\em When $|r-s|=1$ and $(n,r+s)=1$, computer experiments using GAP~\cite{GAP} indicate that the order $|H(r,n,s)^\mathrm{ab}|$ is often a product of large primes so straightforward quotient methods, such as those employed in the proof of Lemma~\ref{lem:mapstoZ2^t}, are unlikely to suffice for proving Conjecture~\ref{conj:noperfect} in general. Since $|F(2,n)^\mathrm{ab}|=|H(2,n,1)^\mathrm{ab}|$ is increasing in $n$ (e.g.~\cite{Conway67}), one might hope to be able to prove Conjecture~\ref{conj:noperfect} by showing that for any $r\geq 2$ the order $|H(r,n,r-1)^\mathrm{ab}|$ is increasing in $n$; however, this is not the case since, for example, $|H(3,5,2)^\mathrm{ab}|=16$ and $|H(3,6,2)^\mathrm{ab}|=13$.
\end{remark}

We also note the following corollary to Lemma~\ref{lem:mapstoZ2^t} which generalizes~\cite[Lemma~4]{CampbellRobertson75} (which deals with the case $r+s\equiv 0$~mod~$n$). It follows immediately from Theorem~9(i) of~\cite{JWW74} which states that if a group $G$ defined by a balanced presentation is finite then $d(G^\mathrm{ab})\leq 3$.

\begin{corollary}\label{cor:improvesCRLem4}
If either
\begin{itemize}
  \item[(a)] $(r+s)/(n,r+s)$ is even and $(n,r+s)\geq 4$; or
  \item[(b)] $(r+s)/(n,r+s)$ is odd and $(n,r+s)\geq 5$  
\end{itemize}
then $H(r,n,s)$ is infinite.
\end{corollary}

\begin{lemma}\label{lem:r-s=2easy}
Let $r,s\geq 1$, $n\geq 2$ and suppose that $(r,n,s)=2$. If $|r-s|\neq 2$ then $d(H(r,n,s)^\mathrm{ab})\geq 2$, and hence $H(r,n,s)$ is not a connected LOG group.
\end{lemma}

\begin{proof}
The abelianisation $H(r,n,s)^\mathrm{ab}$ maps onto
\begin{alignat*}{1}
A
&=\onetwopres{x_0,\ldots,x_{n-1}}
{\prod_{j=0}^{r-1}x_{i+j}=\prod_{j=0}^{s-1}x_{i+j+r}\ (0\leq i<n)}
{x_{2j}=x_0, x_{2j+1}=x_1\ (0\leq j<n/2)}^\mathrm{ab}
\\
&=\pres{x_0,x_1}{(x_0x_1)^{(r-s)/2}}^\mathrm{ab}\\
&\cong \Z_{|r-s|/2}\oplus \Z
\end{alignat*}
which is non-cyclic if $|r-s|\neq 2$, and hence $d(H(r,n,s)^\mathrm{ab})\geq 2$.
\end{proof}

\pagebreak

\begin{lemma}\label{lem:r-s=2hard}
Let $r,s\geq 1$, $n\geq 2$ and suppose that $(r,n,s)=2$ and $|r-s|=2$.
\begin{itemize}
  \item[(a)] If $H(r/2,n/2,s/2)^\mathrm{ab}\neq 1$ then $d(H(r,n,s)^\mathrm{ab})\geq 2$.
  \item[(b)] The group $H(r,n,s)$ is not a 2-generator knot group.
\end{itemize}
\end{lemma}

\begin{proof}
Since $H(r,n,s)\cong H(s,n,r)$ we may assume $r>s$, and so $r=s+2$. Let $R=r/2,N=n/2$ and let $H=H(r,n,s)=H(2R,2N,2R-2)$. Then
\begin{alignat}{1}
H&=\pres{x_i}{\prod_{\alpha=0}^{2R-1}x_{i+\alpha}=\prod_{\alpha=0}^{2R-3}x_{i+2R+\alpha}\ (0\leq i<2N)}\nonumber\\
&=\onetwopres{x_i}{\prod_{\beta=0}^{R-1}x_{i+2\beta}x_{i+2\beta+1}}{\quad =\prod_{\beta=0}^{R-2}x_{i+2(R+\beta)}x_{i+2(R+\beta)+1}\ (0\leq i< 2N)}\nonumber\\
&=\twothreepres{x_{2j},}{x_{2j+1}}
{\prod_{\beta=0}^{R-1}x_{2(j+\beta)}x_{2(j+\beta)+1}=\prod_{\beta=0}^{R-2}x_{2(j+R+\beta)}x_{2(j+R+\beta)+1}}
{\prod_{\beta=0}^{R-1}x_{2(j+\beta)+1}x_{2(j+\beta+1)}}{\quad =\prod_{\beta=0}^{R-2}x_{2(j+R+\beta)+1}x_{2(j+R+\beta+1)}\ (0\leq j<N)}\nonumber\\
&=\threethreepres{x_{2j},}{x_{2j+1},}{y_j,z_j}
{\prod_{\beta=0}^{R-1}y_{j+\beta}=\prod_{\beta=0}^{R-2}y_{j+R+\beta},}
{\prod_{\beta=0}^{R-1}z_{j+\beta}=\prod_{\beta=0}^{R-2}z_{j+R+\beta},}
{y_j=x_{2j}x_{2j+1}, z_j=x_{2j+1}x_{2j+2}\ (0\leq j<N)}\nonumber\\
&=\threethreepres{x_{2j},}{x_{2j+1},}{y_j,z_j}
{\prod_{\beta=0}^{R-1}y_{j+\beta}=\prod_{\beta=0}^{R-2}y_{j+R+\beta},}
{\prod_{\beta=0}^{R-1}z_{j+\beta}=\prod_{\beta=0}^{R-2}z_{j+R+\beta},}
{y_j=x_{2j}x_{2j+1}, x_{2j+1}=z_jx_{2j+2}^{-1}\ (0\leq j<N)}\nonumber\\
&=\twothreepres{x_{2j},}{y_j,z_j}
{\prod_{\beta=0}^{R-1}y_{j+\beta}=\prod_{\beta=0}^{R-2}y_{j+R+\beta},}
{\prod_{\beta=0}^{R-1}z_{j+\beta}=\prod_{\beta=0}^{R-2}z_{j+R+\beta},}
{y_j=x_{2j}z_jx_{2j+2}^{-1}\ (0\leq j<N)}\label{eq:amalgfp}\\
&=\twothreepres{x_{2j},}{z_j}
{\prod_{\beta=0}^{R-1} x_{2(j+\beta)}z_{j+\beta}x_{2(j+\beta)+2}^{-1}}{\quad = \prod_{\beta=0}^{R-2}x_{2(j+R+\beta)}z_{j+R+\beta}x_{2(j+R+\beta)+2}^{-1},}
{\prod_{\beta=0}^{R-1}z_{j+\beta}=\prod_{\beta=0}^{R-2}z_{j+R+\beta}\ (0\leq j<N)}\nonumber
\\
&=\twothreepres{x_{2j},}{z_j}
{x_{2j}\left(\prod_{\beta=0}^{R-1} z_{j+\beta}\right)x_{2(j+R)}^{-1}}{\quad = x_{2(j+R)}\left(\prod_{\beta=0}^{R-2}z_{j+R+\beta}\right)x_{2(j+2R-1)}^{-1},}
{\prod_{\beta=0}^{R-1}z_{j+\beta}=\prod_{\beta=0}^{R-2}z_{j+R+\beta}\ (0\leq j<N)}.\nonumber
\\
\intertext{By setting each $x_{2j}=x_0$ we have that $H$ maps onto}
Q
&=\pres{x_0,z_j}
{\prod_{\beta=0}^{R-1}z_{j+\beta}=\prod_{\beta=0}^{R-2}z_{j+R+\beta}\ (0\leq j<N)}\nonumber
\\
&\cong \pres{x_0}{}*\pres{z_j}
{\prod_{\beta=0}^{R-1}z_{j+\beta}=\prod_{\beta=0}^{R-2}z_{j+R+\beta}\ (0\leq j<N)}\nonumber
\\
&\cong \Z*H(R,N,R-1)=\Z*H(r/2,n/2,s/2).\nonumber
\end{alignat}
Therefore $H(r,n,s)^\mathrm{ab}$ maps onto $\Z\oplus H(r/2,n/2,s/2)^\mathrm{ab}$, so if \linebreak $H(r/2,n/2,s/2)^\mathrm{ab}\neq 1$ this is non-cyclic so $d(H(r,n,s)^\mathrm{ab})\geq 2$, proving part~(a). For part~(b), suppose for contradiction that $H(r,n,s)$ is a 2-generator knot group; then in particular $d(H(r,n,s)^\mathrm{ab})=1$ so by part~(a) we have that $H(r/2,n/2,s/2)^\mathrm{ab}=1$. If $H(r/2,n/2,s/2)$ is non-trivial then it is non-cyclic so $Q$, and hence $H(r,n,s)$, is not a 2-generator group, a contradiction. Therefore $H(r/2,n/2,s/2)=1$ and it follows from~(\ref{eq:amalgfp}) that $H(r,n,s)=\pres{x_{2j}}{x_{2j}=x_{2j+2}\ (0\leq j<N)}\cong \Z$, a contradiction.
\end{proof}

\section{Proof of Theorem~\ref{thm:connLOGclassification} and Corollary~\ref{cor:2genknot}}\label{sec:mainthm}

First observe the following:

\begin{lemma}\label{lem:Zgroup}
Let $r,s\geq 1$, $n\geq 2$ and suppose $r\equiv 0$~mod~$n$ or $s\equiv 0$~mod~$n$. Then
\[ H(r,n,s)\cong \Z_{|r-s|/(n,r-s)}*\underbrace{\Z*\ldots *\Z}_{(n,r-s)-1}.\]
Hence $H(r,n,s)$ is a connected LOG group if and only $(r,n,s)=2$ and $|r-s|=2$, in which case $H(r,n,s)\cong \Z$.
\end{lemma}

\begin{proof}
Since $H(r,n,s)\cong H(s,n,r)$ we may assume that $r\equiv 0$~mod~$n$. Then the set of relations of $H(r,n,s)$ are $x_i\ldots x_{i+r-1}=x_i\ldots x_{i+s-1}$ ($0\leq i<n$) which, after cancelling, become the set of relations $x_i\ldots x_{i+|r-s|-1}=1$. That is, $H(r,n,s)\cong G_n(x_0x_1\ldots x_{|r-s|-1})$ which by~\cite[Theorem~C]{WilliamsSemigroups} is isomorphic to $$\Z_{|r-s|/(n,r-s)}*\underbrace{\Z*\ldots *\Z}_{(n,r-s)-1}.$$
If $H(r,n,s)$ is a connected LOG group then $H(r,n,s)^\mathrm{ab}\cong \Z$ which happens if and only if $(n,r-s)=2$ and $|r-s|=2$ (equivalently $(r,n,s)=2$ and $|r-s|=2$, since $r \equiv 0$~mod~$n$), in which case $H(r,n,s)\cong \Z$.
\end{proof}

\begin{proof}[Proof of Theorem~\ref{thm:connLOGclassification}]
If $H(r,n,s)$ is a connected LOG group then \linebreak $\beta(H(r,n,s)^\mathrm{ab})=1$ so by Theorem~\ref{thm:betti} we may assume that either $r=s$ and $(r,n)=1$ or $r\neq s$ and $(r,n,s)=2$.

In the first case the result follows from Lemma~\ref{lem:(r,n)=1knot} so assume $r\neq s$ and $(r,n,s)=2$. If we have that $r\equiv 0$~mod~$n$ or $s\equiv 0$~mod~$n$ then the result follows from Lemma~\ref{lem:Zgroup}, so we may assume further that $r\not \equiv 0$~mod~$n$ and $s\not \equiv 0$~mod~$n$. If $|r-s|\neq 2$ then $H(r,n,s)$ is not a connected LOG group by Lemma~\ref{lem:r-s=2easy}, and if $|r-s|=2$ and $H(r/2,n/2,s/2)^\mathrm{ab}\neq 1$ then $H(r,n,s)$ is not a connected LOG group, by Lemma~\ref{lem:r-s=2hard}. Now $r+s\equiv 2$~mod~$4$ so $(r+s)/(n,r+s)$ is odd and by Lemma~\ref{lem:mapstoZ2^t} we may assume $(n,r+s)\leq 2$ so $(n,r+s)=2$. If $\{r,s\}=\{4,2\}$ then $H(r/2,n/2,s/2)\cong F(2,n/2)$ and it is well known (see, for example, \cite{Conway67}) that $F(2,n)^\mathrm{ab}\neq 1$ so we may assume further that $\{r,s\}\neq \{4,2\}$. Thus we have that $r,s\not \equiv 0$~mod~$n$, $(r,n,s)=2$, $|r-s|=2$, $H(r/2,n/2,s/2)^\mathrm{ab}=1$, $(n,r+s)=2$ and $\{r,s\}\neq\{4,2\}$, as in part~(c).
\end{proof}

\begin{proof}[Proof of Corollary~\ref{cor:2genknot}]
If $H=H(r,n,s)$ is a 2-generator knot group then $H$ is a connected LOG group so one of the conclusions (a),(b),(c) of Theorem~\ref{thm:connLOGclassification} hold. In (b) we have that $H(r,n,s)\cong \Z$, a contradiction, and in~(c) $H(r,n,s)$ is not a 2-generator knot group by Lemma~\ref{lem:r-s=2hard}.
\end{proof}

\section*{Acknowledgements}

I thank Bill Bogley for helpful comments on a draft of this article.


\begin{thebibliography}{10}

\bibitem{Bleiler94}
Steven~A. {Bleiler}.
\newblock {Two-generator cable knots are tunnel one.}
\newblock {\em {Proc. Am. Math. Soc.}}, 122(4):1285--1287, 1994.

\bibitem{BleilerJones}
Steven~A. {Bleiler} and Amelia~C. {Jones}.
\newblock {On two generator satellite knots.}
\newblock {\em {Geom. Dedicata}}, 104:1--14, 2004.

\bibitem{BogleyWilliams16}
William~A. {Bogley} and Gerald {Williams}.
\newblock {Efficient finite groups arising in the study of relative
  asphericity.}
\newblock {\em {Math. Z.}}, 284(1):507--535, 2016.

\bibitem{BogleyWilliamsCoherence}
William~A. {Bogley} and Gerald {Williams}.
\newblock Coherence, subgroup separability, and metacyclic structures for a
  class of cyclically presented groups.
\newblock {\em J.Algebra}, 480:266--297, 2017.

\bibitem{Brunner77}
A.M. {Brunner}.
\newblock {On groups of Fibonacci type.}
\newblock {\em {Proc. Edinb. Math. Soc., II. Ser.}}, 20:211--213, 1977.

\bibitem{CampbellRobertson75}
C.M. {Campbell} and E.F. {Robertson}.
\newblock {On a class of finitely presented groups of Fibonacci type.}
\newblock {\em {J. Lond. Math. Soc., II. Ser.}}, 11:249--255, 1975.

\bibitem{CampbellThomas86}
Colin~M. {Campbell} and Richard~M. {Thomas}.
\newblock {On infinite groups of Fibonacci type.}
\newblock {\em {Proc. Edinb. Math. Soc., II. Ser.}}, 29:225--232, 1986.

\bibitem{CRS03}
Alberto {Cavicchioli}, Du\v{s}an {Repov\v{s}}, and Fulvia {Spaggiari}.
\newblock {Topological properties of cyclically presented groups.}
\newblock {\em {J. Knot Theory Ramifications}}, 12(2):243--268, 2003.

\bibitem{Conway65}
J.~H. Conway.
\newblock Advanced problem 5327.
\newblock {\em Amer. Math. Monthly}, 72:915, 1965.

\bibitem{Conway67}
J.H. {Conway~et.al.}
\newblock Solution to advanced problem 5327.
\newblock {\em Amer. Math. Monthly}, 74:91--93, 1967.

\bibitem{Cremona08}
J.E. {Cremona}.
\newblock {Unimodular integer circulants.}
\newblock {\em {Math. Comput.}}, 77(263):1639--1652, 2008.

\bibitem{GAP}
The GAP~Group.
\newblock {\em {GAP -- Groups, Algorithms, and Programming, Version 4.8.6}},
  2016.

\bibitem{Gilbert17}
N.D. Gilbert.
\newblock {L}abelled {O}riented {G}raph groups and crossed modules.
\newblock {\em Archiv der Mathematik}, 108:365--371, 2017.

\bibitem{GilbertHowie95}
N.D. {Gilbert} and James {Howie}.
\newblock {LOG groups and cyclically presented groups.}
\newblock {\em {J. Algebra}}, 174(1):118--131, 1995.

\bibitem{GilbertPorter}
N.D. {Gilbert} and T.~{Porter}.
\newblock {\em {Knots and surfaces.}}
\newblock Oxford: Oxford Univ. Press, 1995.

\bibitem{AcunaGordonSimon10}
F.~{Gonz\'alez-Acu\~na}, C.McA. {Gordon}, and J.~{Simon}.
\newblock {Unsolvable problems about higher-dimensional knots and related
  groups.}
\newblock {\em {Enseign. Math. (2)}}, 56(1-2):143--171, 2010.

\bibitem{HarlanderRosebrock15}
Jens {Harlander} and Stephan {Rosebrock}.
\newblock {Aspherical word labeled oriented graphs and cyclically presented
  groups.}
\newblock {\em {J. Knot Theory Ramifications}}, 24(5), 7 pages, 2015.

\bibitem{Howie85}
James {Howie}.
\newblock {On the asphericity of ribbon disc complements.}
\newblock {\em {Trans. Am. Math. Soc.}}, 289:281--302, 1985.

\bibitem{HowieWilliams12}
James {Howie} and Gerald {Williams}.
\newblock {Tadpole labelled oriented graph groups and cyclically presented
  groups.}
\newblock {\em {J. Algebra}}, 371:521--535, 2012.

\bibitem{Ingleton56}
A.W. {Ingleton}.
\newblock {The rank of circulant matrices.}
\newblock {\em {J. Lond. Math. Soc.}}, 31:445--460, 1956.

\bibitem{Johnson80}
D.L. {Johnson}.
\newblock {\em {Topics in the theory of group presentations}}, volume~42 of
  {\em London Mathematical Society Lecture Note Series}.
\newblock Cambridge University Press, 1980.

\bibitem{JWW74}
D.L. {Johnson}, J.W. {Wamsley}, and D. {Wright}.
\newblock {The Fibonacci groups}.
\newblock {\em {Proc. Lond. Math. Soc.}}, 29:577-592, 1974.

\bibitem{Neuwirth}
Lee~Paul {Neuwirth}.
\newblock {Knot groups.}
\newblock {(Annals of Mathematics Studies. Nr. 56.) Princeton, N. J.: Princeton
  University Press}, 1965.

\bibitem{Newman83}
Morris {Newman}.
\newblock {Circulants and difference sets.}
\newblock {\em {Proc. Am. Math. Soc.}}, 88:184--188, 1983.

\bibitem{Ozawa16}
Makoto Ozawa.
\newblock Knots and surfaces.
\newblock {\em preprint arXiv:1603.09039}, 2016.
\newblock (English translation of Musubime to Kyokumen, \em Math. Soc. Japan,
  Sugaku, \em {67(4)}, 403--423 (2015)).

\bibitem{Prischcepov95}
Matve{\u{\i}}~I. {Prishchepov}.
\newblock {Aspherisity, atorisity and symmetrically presented groups.}
\newblock {\em {Commun. Algebra}}, 23(13):5095--5117, 1995.

\bibitem{Scharlemann84}
Martin {Scharlemann}.
\newblock {Tunnel number one knots satisfy the Poenaru conjecture.}
\newblock {\em {Topology Appl.}}, 18:235--258, 1984.

\bibitem{Simon80}
Jonathan {Simon}.
\newblock {Wirtinger approximations and the knot groups of $F\sp n$ in
  $S\sp{n+2}$.}
\newblock {\em {Pac. J. Math.}}, 90:177--190, 1980.

\bibitem{Stallings62}
J.R. {Stallings}.
\newblock {On the recursiveness of sets of presentations of 3-manifold groups.}
\newblock {\em {Fundam. Math.}}, 51:191--194, 1962.

\bibitem{SV00}
Andrzej {Szczepa\'nski} and Andrei {Vesnin}.
\newblock {On generalized Fibonacci groups with an odd number of generators.}
\newblock {\em {Commun. Algebra}}, 28(2):959--965, 2000.

\bibitem{SzczepanskiVesnin01}
Andrzej {Szczepa\'nski} and Andrei {Vesnin}.
\newblock {HNN extension of cyclically presented groups.}
\newblock {\em {J. Knot Theory Ramifications}}, 10(8):1269--1279, 2001.

\bibitem{Williams12}
Gerald {Williams}.
\newblock {Largeness and SQ-universality of cyclically presented groups.}
\newblock {\em {Int. J. Algebra Comput.}}, 22(4):1250035, 19, 2012.

\bibitem{WilliamsSemigroups}
Gerald {Williams}.
\newblock {Fibonacci type semigroups.}
\newblock {\em {Algebra Colloq.}}, 21(4):647--652, 2014.

\end{thebibliography}
\end{document}